\documentclass[12pt]{article}

\usepackage{fancyhdr, amsmath, amssymb, amsfonts, url, subfigure, dsfont, mathrsfs, graphicx, epsfig, subfigure, amsthm}

    \usepackage{epsfig}
    \usepackage{graphicx}
    \usepackage{color}

\newlength{\guillotine}
\newtheorem{thm}{Theorem}[section]

\newtheorem{lemma}[thm]{Lemma}

\newtheorem{example}[thm]{Example}

\theoremstyle{remark}
\newtheorem{rem}[thm]{Remark}
\def\phi{\varphi}
\def\rho{\varrho}
\def\epsilon{\varepsilon}
\begin{document}
%\title{xxxx}
\date{}

\title{
%Bernoulli measures for countble state systems
Maximizing dimension for Bernoulli measures \\ and the Gauss map
}
\author{ Mark Pollicott\thanks{The author was supported by ERC grant ``Resonances''.  The author would like to thank the referee for his many helpful comments.}}

\maketitle

\begin{abstract}
We give a short proof that  there exists a countable state Bernoulli measure maximizing the dimension of their images under the continued fraction expansion.
\end{abstract}

\section{Introduction}
Let $T: [0,1) \to [0,1)$ be the usual Gauss map defined by
$$
T(x) = 
\begin{cases}
\frac{1}{x}\quad  \hbox{\rm (mod $1$)} & \hbox{ if } 0 < x < 1\cr
0 & \hbox{ if } x=0.\cr
\end{cases}
$$
For each   infinite probability vector in 
$$ \mathcal P = \left\{
\underline p = (p_k)_{k=1}^\infty \in [0,1]^\mathbb N\hbox{ : } \sum_{k=1}^\infty p_k = 1
\right\}$$ we can associate a
 natural  $T$-invariant measure $\mu_{\underline p} := \nu_{\underline p}\pi^{-1}$,
where  $\nu_{\underline p}$ 
 is the usual
 countable state Bernoulli measure
 on $\mathbb N^{\mathbb Z}$
and  $\pi: \mathbb N^{\mathbb N} \to [0,1)$ is the usual continued fraction expansion 
$\pi(x_n) = [x_1, x_2, x_3, \cdots]$.
We can define the {\it dimension} of the measure $\mu_{\underline p}$ by
$$
d(\mu_{\underline p}) := \inf \left\{\dim_H(B) \hbox{ : } \hbox{\rm $B$ is a Borel set with }\mu_{\underline p}(B)=1\right\}
$$
where $\dim_H(B)$ denotes  the Hausdorff dimension of $B$ (see \cite{Falconer}, p. 229).
We  define the {\it entropy} and {\it Lyapunov exponents} of the measure $\mu_{\underline p}$
by
$$
h(\mu_{\underline p}) = -\sum_{k=1}^\infty p_k \log p_k
\hbox{ and }
\lambda(\mu_{\underline p}) 
= \int \log |T'| d \mu_{\underline p}(x),
$$
whenever they are finite, 
and then we can write 
$d(\mu_{\underline p}) 
= \frac{h(\mu_{\underline p})}{\lambda(\mu_{\underline p})} 
> 0$.  
Kifer, Peres and Weiss \cite{KPW} observed 
that $d(\mu_{\underline p})$  is uniformly bounded away 
from $1$ (making use of a thermodynamic approach of Walters)
\footnote{In \cite{KPW} they showed $D < 1 - 10^{-7}$, but in unpublished work Jenkinson and the author have improved this to $D < 1 - 5 \times 10^{-5}$}
i.e., 
$$
D:= \sup\left\{
%\frac{h(\mu_{\underline p})}{\lambda(\mu_{\underline %p})} 
d(\mu_{\underline p})
\hbox{ : } 
\underline p \in \mathcal P 
\right\} < 1. \eqno(1.1)
$$
We will give a simple proof of the following result.

\begin{thm}[Exact dimensionality]
\label{main}
There exists  
${\underline p^\dagger} \in \mathcal P$ 
%with $h(\mu_{\underline p^\dagger}), 
%\lambda(\mu_{\underline p^\dagger}) <+\infty$ 
such that:
\begin{enumerate}
\item
$d(\mu_{{\underline p}^\dagger})= D$, i.e., 
$\underline p^\dagger$ realises the supremum in (1.1);
\item   $p_k^{\dagger} \asymp k^{-2 D}$, i.e., $\exists c>1$ such that  
 $\frac{1}{ck^{2D}}
 \leq  
p_k^{\dagger}
 \leq 
  \frac{c}{k^{2D}}$,  for $k \geq 1$;
 and 
\item 
$\mu_{{\underline p}^\dagger}$ is ergodic.
\end{enumerate}
\end{thm}

The first part of the theorem  answers a question the author  was asked by K. Burns.
\footnote{At the {\it Workshop on Hyperbolic Dynamics} (Trieste, 19-23 June 2017)}
I posed the question to my graduate student N. Jurga
who, in collaboration with my PDRA S. Baker,  gave an elementary proof. 
 Their  proof is based on an iterative  construction of a sequence of measures
 $\mu_{\underline p_n}$ with increasing dimension $d(\mu_{\underline p_n})$ by redistributing the weights in the probability vectors $\underline p_n$.
In contrast, the  proof presented below uses  the classical method of Lagrange multipliers on finite dimensional subsets of $\mathcal P$, before taking a limit, and has the merits  of being  short and easy to generalize. 
Part 2.  appears to be new.
%The key ingredient is to compute the derivatives of the dimension. 

\section{Proof of Theorem \ref{main}}
We can begin with the following  standard lemma (see  \cite{Fan}, Lemma 3.2, based on   \cite{KP}).
\begin{lemma}\label{lemma}
If $d(\mu_{\underline p}) > \frac{1}{2}$ then 
$h(\mu_{\underline p}), 
\lambda(\mu_{\underline p}) <+\infty$.
\end{lemma}
%Furthermore, $h(\mu_{\underline p}), 
%\lambda(\mu_{\underline p})$ are real analytic in $%\underline p$.

\noindent
Since it is easy to exhibit 
 $\underline p \in \mathcal P$ with
$
h(\mu_{\underline p}), 
\lambda(\mu_{\underline p})
 <+\infty$ and  
  $
  %\frac{h(\mu_{\underline p})}{\lambda(\mu_{\underline p})}
  d(\mu_{\underline p}) > \frac{1}{2}$  we can 
  deduce $D > \frac{1}{2}$ and 
  use  Lemma \ref{lemma}  to write
  $$
D
=\sup
\left\{
%\frac{h(\mu_{\underline p^*})}{\lambda(\mu_{\underline p^*})} 
\frac{h(\mu_{\underline p})}{\lambda(\mu_{\underline p})} 
\hbox{ : } {{\underline p} \in \mathcal P}
\right\}.
$$
Moreover,  we can approximate  any  $\underline  p \in \mathcal P$  
with $
h(\mu_{\underline p}), 
\lambda(\mu_{\underline p})
 <+\infty$ 
 by a probability vector $\underline p^*$
 with:
 $$
 p_k^* = 
 \begin{cases}
 p_1 + \epsilon_n &\mbox{ if } k=1 \mbox{ (where $\epsilon_n:=\sum_{l=n+1}^\infty p_l$ )}\\
p_k &\mbox{ if } 2 \leq k \leq n \\
 0 &\mbox{ if  } k > n
 \end{cases}
 $$
%with only finitely many non-zero terms 
so that 
$\frac{h(\mu_{\underline p^*})}{\lambda(\mu_{\underline p^*})} $ is arbitrarily close to 
$\frac{h(\mu_{\underline p})}{\lambda(\mu_{\underline p})} $ for $n$ sufficiently large.  
For the entropy,  we have 
$$|h(\mu_{\underline p}) - h(\mu_{{\underline p^*}})| \leq |p_1\log p_1 - (p_1+\epsilon_n)\log(p_1+\epsilon_n)| + \sum_{k=n+1}^\infty p_k |\log p_k| \to 0 \hbox{ as $n \to +\infty$.}$$  
For the Lyapunov exponent, 
let $\epsilon > 0$
and $\log_M|T'(x)| := \min\{\log|T'(x)|, 2\log M\}$ 
then 
$$\left|
%\int \log|T'(x)|d\mu_{\underline p}(x) 
\lambda(\mu_{\underline p})
- \int \log_M|T'(x)|d\mu_{\underline p}(x)\right| \leq 2 \int_0^{1/M} \log \left(\frac{1}{xM} \right) d\mu_{\underline p}(x) < \epsilon$$
 for $M \in \mathbb N$ sufficiently large
(since $\sum_{n=1}^\infty p_n \log n \leq  \lambda(\mu_{\underline p}) <+\infty$)
  and there is a  corresponding inequality with $\underline p^*$ replacing $\underline p$.   
 We can next  bound 
 $$\left|\int \log_M|T'(x)|d\mu_{\underline p}(x) - \sum_{\underline i \in \mathbb N^N}  p_{\underline i}\log_M|T'(x_{\underline i})| \right| < \epsilon$$
  for $N$ sufficiently large, 
where  
$\underline i = (i_1, \cdots, i_N)$  gives the finite continued fraction $x_{\underline i} = [i_1, \cdots, i_N]$ 
and $p_{\underline i} = p_{i_1} \cdots p_{i_N}$, 
and again there is a  corresponding  inequality with $\underline p^*$ replacing $\underline p$.
% (and similarly with $\underline p^*$ replacing $\underline p$).
Finally, we can bound 
$$
\left|\sum_{|\underline i|=N} p_{\underline i}\log_M|T'(x_{\underline i})| - \sum_{|\underline i|=N}  p_{\underline i}^*\log_M|T'(x_{\underline i})|\right|
\leq \log M \sum_{|\underline i|=N} |p_{\underline i} - p_{\underline i}^*| < \epsilon
$$
for $n$ sufficiently large.  (For the last inequality  first note that  for
those $\underline i$ with  $2 \leq i_j \leq n$ for $1\leq j\leq N$ then
$p_{\underline i} = p_{\underline i}^*$ and there is no contribution.
Furthermore, for  those terms with $\underline i$ for which there exists $1\leq
j \leq N$ with $i_j  > n$ the summation can be bounded by $(\epsilon_n + p_1 +
\cdots + p_n)^N - (p_1 + \cdots + p_n)^N \to 0$ as $n\to+\infty$.  Finally, the
remaining part of the summation comes from $\underline i$  with $ i_j \leq n$ for $1\leq j\leq N$ and at least one term being equal $1$, and this 
is $O(\epsilon_n)$.)
 The triangle inequality gives $|\lambda(\mu_{\underline p}) - \lambda (\mu_{\underline p^*})| < 5\epsilon$.
Therefore,   we can also write 
$$
D
= \sup_{n}\sup
\left\{
%\frac{h(\mu_{\underline p^*})}{\lambda(\mu_{\underline p^*})} 
\frac{h(\mu_{\underline p^*})}{\lambda(\mu_{\underline p^*})} 
\mbox{ : } {{\underline p}^* \in \mathcal P_n}
\right\}, 
\eqno(2.1)
$$
where $\mathcal P_n$ ($n \geq 2$) is the finite dimensional simplex consisting  of the probability vectors 
$\underline p^* = (p_k^*)_{k=1}^\infty$ satisfying $p^*_k=0$, for $k>n$.
%In particular, for any $\epsilon > 0$ we can approximate such $%\mu_{\underline p}$
%by  finite state Bernoulli measures
%$\mu_{\underline p^*}$ (i.e., ${\underline p^*}_n=0$ for sufficiently %large $n$) so that 
%$$
%d(\mu_{\underline p^*}) 
%= \frac{h(\mu_{\underline p^*})}{\lambda(\mu_{\underline p^*})} > %d(\mu_{\underline p})  - \epsilon.
%$$ 

For each $n \geq 2$ we can extend the definition of 
$d(\mu_{\underline p^*}):=h(\mu_{\underline p^*})/\lambda(\mu_{\underline p^*})
$ to a  sufficiently small neighbourhood $U_n \supset \mathcal P_n $
so that 
 the function  $U_n \ni \underline p^* \mapsto d(\mu_{p^*})
$  is  well defined and   smooth.
We want to maximize this function  subject to the additional restriction 
$\sum_{k=1}^n p_k^*=1$ which   makes  it natural to 
 use the method of Lagrange multipliers.   This allows us to deduce  that a critical point satisfies
$$
\frac{\partial d(\mu_{\underline p^*})}{\partial p_i^*}
= 
\frac{\partial d(\mu_{\underline p^*})}{\partial p_j^*}
\hbox{ for $i \neq j$.} \eqno(2.2)
$$
 The  logarithmic derivative  of $d(\mu_{\underline p^*})$ 
takes  the form
$$
\frac{1}{d(\mu_{\underline p^*})}
\frac{\partial d(\mu_{\underline p^*})}{\partial p_i^*}
=
\frac{1}{h(\mu_{\underline p^*})}
\frac{\partial h(\mu_{\underline p^*})}{\partial p_i^*}
-
\frac{1}{\lambda(\mu_{\underline p^*})}
\frac{\partial \lambda(\mu_{\underline p^*})}{\partial p_i^*}
\hbox{ for } 1 \leq i \leq n. \eqno(2.3)
$$
We can rewrite the  right hand side of (2.3) using the following two lemmas.
The first follows directly from the definition of $h(\mu_{\underline p^*})$.

\begin{lemma}\label{lemmaone}
$
%\frac{1}{h(\mu_{\underline p^*})}
\frac{\partial h(\mu_{\underline p^*})}{\partial p_i^*}
= 
%\frac{(
-(\log p_i^* + 1)
%)}{h(\mu_{\underline p^*})}
$.
\end{lemma}

We denote the Cantor sets  $E_n:= \{[x_1, x_2, x_3, \ldots ] \hbox{ : } x_1, x_2, x_3, \cdots \leq n\} \subset [0,1]$,  for $n \geq 2$.
For any H\"older continuous
function   $f: E_n \to \mathbb R$ we can define the {\it pressure} (restricted to $E_n$)  by 
$$
P(f) = \sup\left\{
h(\mu) + \int f d\mu \hbox{ : } \mu \hbox{ is a $T$-invariant probability measure supported on $E_n$}
\right\},
$$
where $h(\mu)$ is the entropy of the measure $\mu$, and there is  a unique measure $\mu_f$  realizing the supremum which 
is called the {\it equilibrium state for $f$}.
\begin{example}
We denote the intervals 
 $[i] := \left[\frac{1}{i+1}, \frac{1}{i}\right] \subset (0,1]$, for $i \geq 1$.
%where $\chi_{[i]}}$ is the indicator function for the cylinder $[i] = \{\underline x = (x_n)_{n=1}^\infty \in \Sigma \hbox{ : } x_1=i\}$ and 
Then  $\mu_{\underline p^*}$ is the  equilibrium state for  $f_{\underline p^*} = \sum_{j=1}^n \chi_{[j]} \log p_j^*$.
\end{example}

 For H\"older continuous functions $f,g : E_n \to \mathbb R$ we have that $\mathbb R \ni t \mapsto P(f + tg) \in \mathbb R$ is smooth and 
$$
\frac{\partial P(f + t g)}{\partial t}|_{t=0} = \int g d\mu_f \eqno(2.4)
$$
 (see \cite{ruelle}, Question 5 (a) p.96 and \cite{PP}, Proposition 4.10).  
 For H\"older continuous functions $f,g_1,g_2 : E_n \to \mathbb R$ we have that $\mathbb R^2 \ni (t,s) \mapsto P(f + tg_1 + sg_2) \in \mathbb R$ is smooth and 
$$
\frac{\partial^2 P(f + t g_1 + sg_2)}{\partial t \partial s}|_{s=t=0} = \int (g_1-\overline {g_1})
(g_2 - \overline {g_2})  d\mu_f  + 2\sum_{n=1}^\infty \int (g_1-\overline {g_1}) ( g_2 - \overline {g_2} )
\circ \sigma^nd\mu_f\eqno(2.5)
$$
where we denote  $\overline {g_1} = \int g_1 d\mu $
and $\overline {g_2}  =  \int g_2 d\mu$
  (see \cite{ruelle}, Question 5 (b) p.96
and \cite{PP}, Proposition 4.11).

\medskip 
%We denote the intervals
% $[i] := [\frac{1}{i+1}, \frac{1}{i}]$, for $i \geq 1$.
\begin{lemma}\label{lemmatwo}
$
\frac{1}{\lambda(\mu_{\underline p^*})}
\frac{\partial \lambda(\mu_{\underline p^*})}{\partial p_i^*}
=
\frac{1}{p_i^*}
 \frac{\int_{[i]} 
\log |T'|
d\mu_{\underline p^*}}{\int
\log |T'|
d\mu_{\underline p^*}} - 1
%   \frac{1}{p_i^*}
%\int 
%\left( \chi_{i} - p_i^*\right)
% \left(-  \frac{\log |T'|}{\int
% \log |T'|d\mu_{\underline p^*}} + 1\right)
%d\mu_{\underline p}$.
$
\end{lemma}

\begin{proof}
%This is an elementary  exercise in properties of %transfer operators and the pressure function (see
Using (2.4)  and the definition of $\lambda(\mu_{\underline p^*})$   we  can first rewrite 
$$
 \lambda(\mu_{\underline p^*})
  =  \frac{\partial P( f_{\underline p^*}  + t \log |T'|)}{\partial t}|_{t=0}
\hbox{ and }
\frac{\partial \lambda(\mu_{\underline p^*})}{\partial p_i^*} =  \frac{\partial^2 P( f_{\underline p^*} + 
s\chi_{[i]}/p_i^*  + t \log |T'|)}{\partial s\partial t}|_{t=0, s=0}.
\eqno(2.6)
$$
%where $\chi_{[i]}}$ is the indicator function for the cylinder $[i] = \{\underline x = (x_n)_{n=1}^\infty \in \Sigma \hbox{ : } x_1=i\}$ and 
%$f_{\underline p^*} = -\sum_{j=1}^n \chi_{[j]} \log p_j^*$ is chosen so that  $\mu_{\underline p^*}$ is the  equilibrium state for $f_{\underline p^*}$.
Next  we can use (2.5) 
with $f = f_{\underline p^*}$, $
g_1 = \chi_{[i]}/p_i^*$
and 
 $g_2 =  \log |T'|$
to write 
$$
\begin{aligned}
&\frac{\partial^2 P( f_{\underline p^*} +
s\chi_{[i]}/p_i^*  + t \log |T'|)}{\partial s\partial t}|_{t=0, s=0}\cr 
&=
\frac{1}{p_i^*}
 \int 
\left( \chi_{[i]} - p_i^*\right)
 \left( \log |T'| -  \int
 \log |T'|d\mu_{\underline p^*}\right)
d\mu_{\underline p^*} \cr
&\qquad+
\frac{2}{p_i^*}
\sum_{n=1}^\infty \int 
\left( \chi_{[i]} - p_i^*\right)
 \left(  \log |T'| - \int
 \log |T'|d\mu_{\underline p^*}\right) \circ \sigma^n
d\mu_{\underline p^*}.
\end{aligned}
\eqno(2.7)
$$
If we consider the transfer operator 
$\mathcal L_{f_{\underline p^*}}: C^0(E_n) \to C^0(E_n)$
defined by  
$$
\mathcal L_{f_{\underline p^*}} w(x) = \sum_{k=1}^n p_k^* w\left(\frac{1}{k+x}\right) \eqno(2.8)$$
which is the dual to the Koopman operator
(see \cite{ruelle})
then 
since the dual to the transfer operator satisfies 
%$\mathcal L_{f_{\underline p^*}}1=1$
$\mathcal L_{f_{\underline p^*}}^*\mu_{\underline p^*}=\mu_{\underline p^*}$ 
%and 
%$\mathcal L_{f_{\underline p^*}} \chi_{[i]} =  p_i^*$
we can rewrite (2.7) as
$$
\begin{aligned}
&\frac{\partial^2 P( f_{\underline p^*} +
s\chi_{[i]}/p_i^*  +t \log |T'|)}{\partial s\partial t}|_{t=0, s=0}\cr
&=
\frac{1}{p_i^*}
 \int 
\left( \chi_{[i]} - p_i^*\right)
 \left( \log |T'| - \int
 \log |T'|d\mu_{\underline p^*}\right)
d\mu_{\underline p^*} \cr
&\qquad+
\frac{2}{p_i^*}
\sum_{n=1}^\infty \int 
\mathcal L_{f_{\underline p^*}}^n
\left( \chi_{[i]} - p_i^*\right)
 \left( \log |T'| -  \int
 \log |T'|d\mu_{\underline p^*}\right)
d\mu_{\underline p^*}.
\end{aligned}
\eqno(2.9)
$$
From the definition of
$\mathcal L_{f_{\underline p^*}}$
we  see that 
$\mathcal L_{f_{\underline p^*}}
\left( \chi_{[i]} - p_i^*\right) = 0
$
and we can deduce that the series in (2.9) vanishes
%$$
%\begin{aligned}
%\frac{\partial^2 P( f_{\underline p^*} -
%s/p_i^*  - t \log |T'|)}{\partial s\partial t}|_{t=0, s=0}
%&=
%\frac{1}{p_i^*}
%\int 
%\left( \chi_{i} - p_i^*\right)
% \left(-  \log |T'| + \int \log |T'|d\mu_{\underline p^*}\right)
%d\mu_{\underline p}\cr
%\end{aligned}
%$$
and then using (2.6) we can write
$$
\begin{aligned}
\frac{1}{\lambda(\mu_{\underline p^*})}
\frac{\partial \lambda(\mu_{\underline p^*})}{\partial p_i}
&=   \frac{1}{p_i^*}
\int 
\left( \chi_{[i]} - p_i^*\right)
 \left(  \frac{\log |T'|}{\int
 \log |T'|d\mu_{\underline p^*}} -  1\right)
d\mu_{\underline p}
%\cr
%&
 =
\frac{1}{p_i^*}
 \frac{\int_{[i]} 
\log |T'|
d\mu_{\underline p}}{\int
\log |T'|
d\mu_{\underline p}} - 1.
%&\approx \log i.
\end{aligned}
$$

\end{proof}

\noindent 
Using  the formulae in  Lemmas \ref{lemmaone} and 
\ref{lemmatwo} and the equality (2.3) 
we can rewrite (2.2)  
as 
$$
- (\log p_i^* + 1)  - d(\mu_{p^*}) \left( \frac{1}{p_i^*} \frac{\int_{[i]} \log |T'| d\mu_{p^*}}{\int \log |T'| d\mu_{p^*}} - 1\right)
= 
- (\log p_j^* + 1)  - d(\mu_{p^*}) \left( \frac{1}{p_j^*} \frac{\int_{[j]} \log |T'| d\mu_{p^*}}{\int \log |T'| d\mu_{p^*}} - 1\right)
$$
for all $1 \leq i,j \leq n$.  
Moreover, since  $2p_i^*\log i \leq \int_{[i]} \log |T'| d\mu_{p^*} \leq 2p_{i}^*\log (i+1)$ this implies that 
$$
2d(\mu_{p^*})  \log \left(\frac{i}{j+1}\right)
\leq 
\log \left(\frac{p_j^*}{p_i^*}\right)
\leq  2 d(\mu_{p^*}) \log \left(\frac{i+1}{j}\right)
\eqno(2.10)
$$
 for any $n \geq 2$ and $n \geq i > j$.
%$$
%\log \left(\frac{p_i^*}{p_j^*}\right)
%= d(\mu_{\underline p^*}) \log \left(\frac{i}{j}\right)
%\left(1 + o(1) \right).
%$$
%
%$\frac{1}{d(\mu_{\underline p^*})}
%\frac{\partial d(\mu_{\underline p^*})}{\partial p_i} 
%\approx \log i$, 
%where the implied constants are independent of $n$.

To construct $\underline p^\dagger \in \mathcal P$ we use a simple tightness argument.
For each sufficiently large  $n$, let   $\underline p^{(n)} \in \mathcal P_n$  denote a  measure maximizing $\mathcal P_n \ni  \underline p^* \mapsto d(\mu_{p^*})$. 
It then follows from (2.1) that   $\lim_{n\to+\infty} d(\underline p^{(n)}) = D > \frac{1}{2}$.  
Since  (2.10) applies to each of the  $\underline p^{(n)}$,
and $d(\mu_{\underline p^{(n)}})$ is arbitrarily close to $D$ for $n$ sufficiently large, 
 one can choose $\epsilon > 0$, $C > 0$ and $n_0 > 0$  such that 
%and $j$
%(with $\liminf p_j^{*(n)} > 0$)
 $$p_k^{(n)} \leq Ck^{-(D - \epsilon)}
 \hbox{  for all $k \geq 1$ and $n \geq n_0$}.
 $$
We can choose a  subsequence  $\underline p^{(n_r)}$ ($r \geq 1$) converging  
(using the usual diagonal argument) to some $\underline p^{\dagger} \in \mathcal P$.    
To see that $d(\mu_{\underline p^\dagger}) = D$  we first observe that 
 $d(\mu_{\underline p^{*(n_r)}}) > \frac{1}{2}$ for sufficiently large $r$.   We can deduce from (2.1),
 and the definitions of the entropies and Lyapunov exponents,
  that $d(\mu_{\underline p^{\dagger}}) = \lim_{r \to +\infty}d(\mu_{\underline p^{*(n_r)}}) = D$.  This completes the proof of part 1. 

 By applying (2.10) to $\underline p^{*(n_r)}$ and taking the limit $r\to +\infty$ shows that
the same bounds apply for $\underline p^\dagger$ (with $d(\mu_{\underline p^\dagger}) = D$) and this completes the proof of part 2.

%nd using the tightness coming from the bounds on $p_i^*$,  we can deduce   that there exists a limit point $\underline p^\dagger \in \mathcal P$   satisfying both $D= d(\mu_{\underline p^\dagger})$ (using (2.1)) and (2.10).
%The proof of  parts 1 and 2 of Theorem \ref{main} follow immediately.

To prove the final part of the theorem,  we use another  standard argument (cf.
\cite{PP}, Chapter 2, for the case of H\"older functions).  We can consider the
associated transfer operator $\mathcal L_{\underline p^\dagger}: C^1([0,1]) \to
C^1([0,1])$ defined (by analogy with (2.8)) as 
$$
\mathcal  L_{f_{\underline p^\dagger}} w(x) = \sum_{k=1}^\infty p^{\dagger}_k w\left( \frac{1}{k+x}\right) \hbox{ for } w \in C^1([0,1]).
$$
First observe that $ \mathcal L_{\underline p^\dagger} 1 = 1$, and thus
$\|\mathcal L_{\underline p^\dagger} w \|_\infty  \leq \|w\|_\infty$,  and 
$\left\| \frac{d}{dx} \left(\mathcal  L_{f_{\underline p^\dagger}}^2 w\right)(x) \right\|_\infty \leq \frac{1}{4}\|w\|_\infty$  (cf. \cite{PP}, Proposition 2.1).
This is sufficient to show that for any  $w \in C^1([0,1])$ one has $\|\mathcal
L_{\underline p^\dagger}^n w - \int w d\mu_{\underline p^\dagger}\|_\infty \to 0$ as $n \to +\infty$ (compare \cite{PP}, Theorem 2.2, (iv)).
We can deduce that for $v,w \in C^1([0,1])$ with $\int v d\mu_{\underline p^\dagger} = 0 = \int v d\mu_{\underline p^\dagger}$ then 
since  $\mathcal L_{\underline p^\dagger}^* \mu_{\underline p^\dagger}=\mu_{\underline p^\dagger}$ we have 
$$\left|\int v\circ T^n w d\mu_{\underline p^\dagger}\right|
= \left|\int v \mathcal L_{\underline p^\dagger}^n  w d\mu_{\underline p^\dagger} \right|
\leq \|v\|_\infty.  \| \mathcal L_{\underline p^\dagger}^nw\|_\infty
\to 0 \hbox{ as } n \to +\infty.$$
In particular, this implies that $\mu_{\underline p^\dagger}$ is strong mixing, and thus ergodic (cf. \cite{PP}, Proposition 2.4).

%\begin{lemma}
%There exists $C>0$ so that for all $n$
%there exists $p^*\in \mathcal P_n$ satisfying both 
% $d(\mu_{\underline p^*}) = 
% \sup\{d(\mu_{\underline p}) \hbox{ : } \underline p \in 
%\mathcal P_n \}$
%and  $ \frac{i^{-2D} }{C} 
%\leq p_i^*  \leq C i^{-2D}$, for all $i \geq 1$.
%\end{lemma}
%Since we can assume that $d(\mu_{\underline %p^*}) \geq \frac{1}{2}$ we can conclude
%that 

\section{Additional remarks}
\begin{rem}
The approach we have described should  apply beyond the Gauss map.   One of the best known classes are the R\'enyi $f$-expansions, where
in the definition of $T$ the function  $1/x$ is replaced by a more general (continuously differentiable)
monotone decreasing  function  $f:(0,1] \to [1, +\infty)$ and $\pi$  is replaced by a map  $\pi_f: \mathbb N^{\mathbb N} \to [0,1]$ (see \cite{schweiger}, chapter 10).
We  can  then  set $s_0 = \inf\{s > 0 \hbox{ : } \sum_{n=1}^\infty |(f^{-1})'(n)|^s < +\infty\}$ and define  $D_f$ by analogy with (2.1).
We would need to assume that $D_f > s_0$
%.    The  analogue of (2.10) is then
%$$
%d(\mu_{p^*}) \log \left( \frac{f'(j+1)}{f'(i)}\right)  \leq \log \left(\frac{p_j^*}{p_i^*} \right) \leq d(\mu_{p^*}) \log \left( \frac{f'(i)}{f'(i+1)}\right)
%$$
%for $n \geq 2$ and $n \geq i > j$.    The previous  proof now gives 
and then one expects that  a version of Theorem \ref{main} holds in  this setting, where part 2. would now take the form $\underline p^\dagger_k \asymp  |(f^{-1})'(k)|^D$, for $k \geq 1$.

%\begin{thm}[Properties]\label{uniqueness}
%As a partial converse to part 2 of Theorem \ref{main}, one can show that any  measure 
%$\mu_{\underline p}$  realising  the supremum in (1.1) has coefficients that 
%decay polynomially fast.  More precisely,  we can associate to $\underline p \in \mathcal P$
%probability  vectors $\underline p^{(n)} \in \mathcal P_n$ defined by $p^{(n)}_k = p_k$ for $1\leq k \leq n-1$
%and $p^{(n)}_n = \sum_{k=n}^\infty p_k$. 
\end{rem}

\begin{rem}
The proof of Theorem \ref{main} also generalizes in the following way.   Let
 $f: (0,1] \to \mathbb R$ be a locally constant function of the form $f(x) = b_i$, say, if for $ \frac{1}{i+1} < x \leq \frac{1}{i}$, with super-polynomial decay (i.e., for any $\beta>0$, $|b_i| = O(i^{-\beta})$).
Given $\alpha \in \hbox{\rm int}\{\int f d\mu_{\underline p} \hbox{ : } \underline p \in \mathcal P\}$ we can  consider an expression common in multifractal analysis:
$$
D_f:= \sup\left\{
\frac{h(\mu)}{\int \log x d\mu}
\hbox{ : } h(\mu_{\underline p}), 
\lambda(\mu_{\underline p}) <+\infty
\hbox{ and } \int f d\mu= \alpha
\right\}. \eqno(3.1)
$$
Then there exists $\underline p^f
= (\underline p^f_k)_{k=1}^\infty \in \mathcal P$ such that:
$\mu_{\underline p^f}$ realises the supremum in (3.1); $p_k^{f} \asymp k^{-2 D_f}$ for $k \geq 1$; and 
 $\mu_{\underline p^f}$ is ergodic. 
\end{rem}

\begin{rem}
We can compute higher derivatives 
of $d(\mu_{\underline p})$ using 
higher derivatives of the pressure function.  However, these do not seem particularly useful.
\end{rem}

\end{document}